\documentclass[10pt]{amsart}

\usepackage{amscd,amsmath,amssymb,amsfonts,verbatim}
\usepackage{graphicx}
\usepackage[active]{srcltx}
\usepackage[cmtip, all]{xy}
\usepackage[pdftex,colorlinks=true,citecolor=black,linkcolor=black,urlcolor=blue]{hyperref}
\usepackage{calc}

\newtheorem{theorem}{Theorem}[section]
\newtheorem{proposition}[theorem]{Proposition}
\newtheorem{lemma}[theorem]{Lemma}
\newtheorem{corollary}[theorem]{Corollary}

\theoremstyle{definition}

\newtheorem{definition}[theorem]{Definition}

\theoremstyle{remark}
\newtheorem{remark}[theorem]{Remark}
\newtheorem*{ack}{Acknowledgments}
\numberwithin{equation}{section}

\newcommand{\sF}{{\mathcal F}}

\newcommand{\sI}{{\mathcal I}}

\newcommand{\sM}{{\mathcal M}}
\newcommand{\sN}{{\mathcal N}}

\newcommand{\sP}{{\mathcal P}}

\newcommand{\sR}{{\mathcal R}}

\newcommand{\bL}{{\mathbb L}}

\newcommand{\Q}{{\mathbb Q}}

\newcommand{\Z}{{\mathbb Z}}

\newcommand{\mc}{\mathcal}

\newcommand{\codim}{{\rm codim}}

\newcommand{\Hom}{{\rm Hom}}

\newcommand{\Spec}{{\rm Spec \,}}

\newcommand{\id}{{\operatorname{id}}}

\newcommand{\Sm}{{\mathbf{Sm}}}
\newcommand{\SmProj}{{\mathbf{SmProj}}}

\newcommand{\can}{{\operatorname{\rm can}}}

\newcommand{\Ab}{{\mathbf{Ab}}}

\newcommand{\eff}{{\operatorname{\rm eff}}}

\newcommand{\ds}{{/\kern-3pt/}}

\newcommand{\lci}{{\rm l.c.i.\!}}

\newcommand{\disp}{\displaystyle}

\newcommand{\by}[1]{\overset{#1}{\to}}
\newcommand{\longby}[1]{\overset{#1}{\longrightarrow}}
\newcommand{\iso}{\by{\sim}}
\renewcommand{\Phi}{\varPhi}

\newcommand{\n}{\mathrm{num}}
\newcommand{\oq}{\Omega^*_\Q}
\newcommand{\bt}{\mathbf{t}}
\newcommand{\bn}{\mathbf{n}}

\newcommand{\bm}{\mathbf{m}}

\newcommand{\Id}{\operatorname{\mathrm{Id}}}
\newcommand{\Aut}{\operatorname{\mathrm{Aut}}}

\renewcommand{\dim}{\text{\rm dim}}
\newcommand{\tuborg}{\left\{\begin{array}{ll}}
\newcommand{\sluttuborg}{\end{array}\right.}

\begin{document}

\title{Fourier-Mukai transformation on algebraic cobordism}
\author{Anandam BANERJEE and Thomas HUDSON}
\address{Department of Mathematical Sciences, KAIST, 291 Daehak-ro, Yuseong-gu,
Daejeon, 305-701, Republic of Korea (South)}
\email{anandam1729@kaist.ac.kr}
\address{Department of Mathematical Sciences, KAIST, 291 Daehak-ro, Yuseong-gu,
Daejeon, 305-701, Republic of Korea (South)}
\email{hudson.t@kaist.ac.kr}


\begin{abstract}We define a notion of Fourier-Mukai transform on algebraic cobordism cycles with $\Q$-coefficients on an abelian variety. We use this to produce a Beauville decomposition of algebraic cobordism and study its consequences, including a decomposition of the cobordism motive of an abelian variety.
\end{abstract}


\maketitle

\section{Introduction}
The Fourier transformation is a well-known operator in analysis that gives an isometry between the $L^2$-spaces of a real vector space and its dual vector space. In \cite{Mukai}, Mukai introduced an analogous notion for sheaves of modules over abelian varieties. 
Let $A$ be an abelian variety and $\hat{A}$ be its dual abelian variety. Using the normalized Poincar\'e bundle on $A\times \hat{A}$, Mukai defined a functor between the derived categories of the sheaves of modules over $A$ and $\hat{A}$, and proved that this is an equivalence of categories. In \cite{Beauville1}, Beauville used his results to define  such a functor on cohomology, K-theory and the Chow ring of $A$ with similar properties in each theory. He studied the Fourier transformation on Chow rings in detail and proved interesting consequences, including a decomposition of the Chow ring of an abelian variety into eigenspaces of the pullback by a multiplication by $n$ morphism, for integers $n$ (see \cite{Beauville2}). Deninger and Murre used Beauville and Mukai's work in \cite{DM} to give a  decomposition of $CH(A\times A)\otimes \Q$ into eigenspaces of $(\id\times n)^*$, which induces a canonical decomposition of Chow motives of an abelian variety. 

The Chow ring being an oriented cohomology theory (see \cite{LM} for the definition), a natural question to ask is whether a functor with the usual properties of a Fourier-Mukai transformation can be defined on any other oriented cohomology theory. Levine and Morel defined the theory of algebraic cobordism in \cite{LM} and showed that it is the universal oriented cohomology theory on $\Sm_k$.  In this paper, we define a Fourier-Mukai operator on the theory $\Omega_\Q$ of algebraic cobordism with $\Q$-coefficients and study its consequences. 

The key idea that helped us extend the definition of the Fourier-Mukai operator to $\Omega_\Q$ is the following theorem:
\begin{theorem}
For any abelian variety $A$  over a field $k$ of characteristic $0$, the canonical morphism of oriented cohomology theories induces an  isomorphism of rings
\[ \psi_A:\Omega^*_\Q(A)\to CH_\Q[\mathbf{t}]^*(A),\]
where $CH_\Q[\mathbf{t}]^*(A)=CH_\Q(A)[t_1,t_2,\ldots]$ is the graded polynomial ring on variables $t_i,i>0$ of degree $-i$. 
\end{theorem}
We also show that $\psi$ commutes with push-forward maps and the pullbacks of morphisms between abelian varieties. The Fourier-Mukai operator on $CH_\Q$ induces one on $CH_\Q[\mathbf{t}]^*$ by extension of scalars. Denoting this operator as $\sF_{CH}$, we show that the Fourier-Mukai operator $\sF_\Omega$ on $\Omega_\Q$ that we defined in \S~\ref{sec:definition} satisfies
\[ \psi_{\hat{A}}\circ\sF^\Omega=\sF^{CH}\circ\psi_A. \]
This commutativity helps us obtain most of the properties of $\sF^\Omega$ in Proposition~\ref{prop:properties}.

We also prove an analogue of Beauville decomposition for algebraic cobordism with $\Q$-coefficients, that is, a decomposition of $\Omega_\Q$ into eigenspaces of $\bn^*$ where $\bn$ in the multiplication by $n$ morphism of an abelian variety.
\begin{theorem}
Let $A$ be an abelian variety of dimension $g$ over $k$. Then, we have
\[ \Omega_\Q^{p}(A) = \bigoplus_{s=2p-2g}^{\mathrm{Min}(2p,p)} \Omega^{\phantom{\Q}p}_{\Q s}(A), \]
where \[ \Omega^{\phantom{\Q}p}_{\Q s}(A):=\big\{ x\in \Omega^{p}_\Q(A) \vert\, \bn^*x=n^{2p-s}x \big\}. \]
\end{theorem}
As an application of the properties of the Fourier-Mukai transformation, we generalize a result of Bloch \cite{Bloch} to the case of cobordism cycles in \S~\ref{sec:num}. Let 
$\sN^*(A)$ be the group of numerically trivial cobordism cycles on $A$ as defined in \cite[Definition~3.1]{BP}  and let $\star$ be the Pontrygin product on $\Omega^*(A)$. Then, we show
\begin{proposition}
$\sN^{*\star (g+1)}_\Q=(0)$.
\end{proposition}
In \S~\ref{sec:motive}, we prove that there is a canonical decomposition of the cobordism motive (defined in \cite[\S~6]{NZ}) of an abelian variety.
\begin{theorem}
There is a canonical decomposition of the cobordism motive of an abelian variety $A$ of dimension $g$ over $k$:
\[ h_\Omega(A)=\bigoplus_{i=0}^{2g}h_\Omega^i(A)\,, \]
where $h_\Omega(A)=(A,\id_A, 0)$ is the motive of $A$, $h_\Omega^i(A)=(A,\pi_i,0)$ and the $\pi_i$'s are such that $c_\Omega(\bn)\circ\pi_i=n^i\pi_i=\pi_i\circ c_\Omega(\bn)$. Furthermore, we have $\tilde{\varphi}(h_\Omega^i(A))=h^i_{CH}(A)^\can$ is the canonical decomposition of the Chow motive of $A$.
\end{theorem}

\section{Oriented cohomology theories and Algebraic cobordism}

 In \cite{LM}, inspired by the work of Quillen on complex differentiable manifolds, Levine and Morel introduced the notion of an oriented cohomology theory: a contravariant functor $A^*$ from $\Sm_k$ to graded rings together with a collection of push-forward maps $f_*$ associated to projective morphisms. This family is meant to respect functoriality and to be compatible with the pull-back morphisms $g^*$ on cartesian squares every time $f$ and $g$ are transverse. Finally, the functor is supposed to satisfy the projective bundle formula, which expresses the evaluation of $A^*$ on a projective bundle in terms of that of the base, together with the extended homotopy property, which requires $p^*:A^*(X)\rightarrow A^*(V)$ to be an isomorphism for every vector bundle $E\rightarrow V$ and every $E$-torsor $p:V\rightarrow X$. As one might expect, a morphism of oriented cohomology theories is a natural transformation of functors which is also compatible with the push-forward morphisms $f_*$.  
  For the precise definition we refer the reader to \cite[Definition~1.1.2]{LM}. Important examples of functors which are also oriented cohomology theories includes the Chow ring $CH^*$ and $K^0[\beta,\beta^{-1}]$, a graded version of the Grothendieck ring of vector bundles. 

    A relevant feature of oriented cohomology theories is that they allow a theory of Chern classes. Even though in order to establish it for any bundle $E\rightarrow X$ it is necessary to rely on the projective bundle formula, the first Chern class of a line bundle $L\rightarrow X$ can be defined by making use only of push-forward and pull-back morphisms: if $s$ denotes the zero section of $L$ one sets $c_1(L):=s^*s_* 1_X\in A^*(X)$ and defines the first Chern class operator $\tilde{c}_1(L):A^*(X)\rightarrow A^{*+1}(X)$ as multiplication by $c_1(L)$. Once the first Chern class is available, one may consider how it relates to the tensor product: given two line bundles $L$ and $M$ over some smooth scheme $X$, what is the relation between $c_1(L)$, $c_1(M)$ and $c_1(L\otimes M)$? While for the Chow group one has $c_1(L\otimes M)=c_1(L)+c_1(M)$, the equality is not true in general for oriented cohomology theories and one is forced to replace the usual addition with a formal group law:
$$c_1(L\otimes M)=F_A(c_1(L), c_1(M))$$
for a certain $F_A\in A^*(k)[[u,v]]$. A commutative formal group law of rank one $(R,F_R)$ is constituted by a ring $R$ and a power series $F_R\in R[[u,v]]$ satisfying conditions which are analogues of those for the operation in a group. For instance, the analogue of the associative property reads
$$F_R(F_R(u,v),w)=F_R(u,F_R(v,w))\in R[[u,v,w]]\ .$$
In \cite{Lazard}, Lazard identified the universal such object $(\bL,F)$ and proved that the ring of coefficients, now known as the Lazard ring, is isomorphic to $\Z[a_1,a_2,\ldots]$. In this context, the universality means that for every formal group law $(R,F_R)$ there exist a unique ring homomorphism $\phi_R:\bL\rightarrow R$ such that $\phi_R(F)=F_R$, where $\phi_R(F)$ stands for the power series obtained by applying $\phi_R$ to the individual coefficients of $F$. Since it will be needed later on, let us add that the Lazard ring can be made into a graded ring $\bL^*$ by setting $\text{deg}\, a_i=-i$.

Taking into consideration formal group laws makes it evident that the analogy with the situation in topology does not end with the introduction of oriented cohomology theories. In fact, in \cite{Quillen}, Quillen proved that complex cobordism $MU^*$ is universal among complex oriented cohomology theories, that $MU^*(pt)\simeq \bL^*$ and finally that its formal group law is the universal one. From this perspective, the theory of algebraic cobordism $\Omega^*$, developed in \cite{LM} by Levine and Morel, represents the exact analogue of $MU^*$ as one has the following theorems:

\begin{theorem}[{\cite[Theorem~1.2.6]{LM}}]
Let $k$ be a field of characteristic $0$. Then, given any oriented cohomology theory $A^*$ on $\Sm_k$, there is a unique morphism
\[ \nu_A:\Omega^*\to A^* \]
of oriented cohomology theories.
\end{theorem}

\begin{theorem}[{\cite[Theorem~1.2.7]{LM}}]
For any field $k$ of characteristic $0$, the canonical homomorphism classifying $F_\Omega$
$$\phi_\Omega: \bL^*\rightarrow \Omega^*(k)$$
is an isomorphism.
\end{theorem}
\noindent Notice that, provided $\Omega^*(k)$ is identified with the Lazard ring through $\phi_\Omega$,  the evaluation of $\nu_A$ on $\Spec k$ coincides with $\phi_A$ and as a consequence one has $\nu_A(F_\Omega)=F_A$. 

We briefly sketch the construction of algebraic cobordism. As a group $\Omega^*(X)$ is obtained from the free group generated by the isomorphism classes of cobordism cycles
$$[f:Y\rightarrow X,L_1,\ldots, L_r]$$
 where $f$ is a projective morphism with $Y\in\Sm_k$ and $\{L_1, \ldots,L_r\}$ is a (possibly empty) family of line bundles over $Y$. Such a cycle has dimension $d=\text{dim}_k Y-r$ and codimension $\text{dim}_k X-d$. On this group one successively imposes three families of relations so that the quotient will satisfy three corresponding axioms (Dim), (Sect) and (FGL). Note that in order to make sense of the last axiom it is necessary to tensor the group obtained from the second quotient by the Lazard ring, so as to have at hand a formal group law. For what concerns the multiplicative structure, it is achieved by constructing pull-backs for $\lci$ morphisms through an adaptation of the method used by Fulton for Chow groups: one relies on the deformation to the normal cone to reduce to the case of a divisor, which is handled separately. 

 Let us finish by mentioning that the main technical tool used in the proofs of the various axioms is represented by the  localization sequence 
$$\Omega_*(Z)\stackrel{i_*}\longrightarrow \Omega_*(X)\stackrel{j^*}\longrightarrow \Omega_*(U)\longrightarrow 0 \ $$
for any closed embedding $i:Z\rightarrow X$ with open complement $j:U\rightarrow X$.

\subsection{Twists of oriented cohomology theory}
We recall from \cite[\S~4.1.8--9 \& \S~7.4.2]{LM} the construction of the twisting of an oriented cohomology theory on $\Sm_k$.  Let $A^*$ be an oriented cohomology theory on $\Sm_k$ and $\tau=(\tau_i)\in\prod_{i=0}^\infty A^{-i}(k)$, with $\tau_0=1$. 
\begin{definition}
The \emph{inverse Todd class operator} of a line bundle $L\to X$ is defined to be the operator on $A^*(X)$ given by the infinite sum
\[\widetilde{Td}_\tau^{-1}(L)=\sum_{i=0}^\infty\tilde{c}_1(L)^i\tau_i.\]
In \cite[Proposition~4.1.20]{LM}, Levine and Morel showed that this can be extended to any vector bundle $E\to X$ in a unique way to give an endomorphism $\widetilde{Td}_\tau^{-1}(E):A^*(X)\to A^*(X)$ of degree 0, such that if $0\to E'\to E\to E''\to 0$ is an exact sequence of vector bundles over $X$, then one has $\widetilde{Td}_\tau^{-1}(E)=\widetilde{Td}_\tau^{-1}(E')\circ\widetilde{Td}_\tau^{-1}(E'')$. Thus, it naturally extends to a map 
\[ \widetilde{Td}_\tau^{-1}:K^0(X)\to \Aut(A^*(X)).\]
In fact, on an oriented cohomology theory on $\Sm_k$, $\widetilde{Td}_\tau^{-1}(E)$ is the multiplication by a class $Td_\tau^{-1}(E):=\widetilde{Td}_\tau^{-1}(E)(1_X)\in A^*(X)$, called the \emph{inverse Todd class} of $E$. For any smooth equidimensional $Y\by{f}X$, it is shown that ${Td}_\tau^{-1}(f^*E)=f^*{Td}_\tau^{-1}(E)$. 
\end{definition}
Suppose $X, Y$ are in $\Sm_k$. Then, any  $f:Y\to X$ is an  $\lci$ morphism. Let $f=q\circ i$ be a factorization such that $i:Y\to P$ is a regular embedding and $q:P\to X$ is smooth. Letting $\sI$ be the ideal sheaf of $Y$ in $P$, we define the normal bundle $N_i$ to be the bundle over $Y$ whose dual has sheaf of sections $\sI/\sI^2$. We let $N_f\in K^0(Y)$ be the class $[N_i] -[i^*T_q]$, where $T_q$ is the relative tangent bundle associated to $q$. For any $\tau$ as above, one may construct an oriented cohomology theory on $\Sm_k$, denoted $A^*_{(\tau)}$, by  twisting the first Chern classes and the pull-back maps. If  $f^*$ and ${c}_1$ are the pull-back and the first Chern class respectively in $A^*$, then $A^*_{(\tau)}(X)=A^*(X)$ as groups and in $A^*_{(\tau)}$,  
\begin{itemize}
\item $f^*_{(\tau)}={Td}_\tau^{-1}(N_f)\cdot f^*$, where ${Td}_\tau^{-1}$ is the inverse Todd class;
\item for any line bundle $L$ over $X$, the first Chern class of $L$ in $A^*_{(\tau)}$ is ${c}_1^{(\tau)}(L)= \sum_{i=0}^\infty c_1(L)^{i+1}\tau_i$;
\item if $\cdot$ denotes the product in $A^*(X)$ and $\cdot_\tau$ denotes the product in $A^*_{(\tau)}(X)$, then $x\cdot_\tau y = Td_\tau^{-1}(N_{\delta_X})\cdot x\cdot y$, for any $x,y\in A^*(X)$, where $\delta_X:X\to X\times X$ is the diagonal morphism.
\end{itemize}
The push-forward maps  are unchanged.

If $f=q\circ i$ is a factorization such that $i:Y\to P$ is a regular embedding and $q:P\to X$ is smooth, then  note that $P$ is smooth over $k$, since $X$ is smooth and $q$ is smooth.  Thus, considering $\xymatrix@=0.4em{Y \ar[rr]^i \ar[dr] && P \ar[dl]\\ & \Spec k & }$, we get by \cite[B.7.2.]{Fulton}, the exact sequence
 \begin{equation}
 0\longrightarrow T_Y\longrightarrow i^*T_P\longrightarrow N_i\longrightarrow 0.
 \end{equation}
 Thus, in $K^0(Y)$, $[T_Y] =[i^*T_P] -[N_i]$. Also, since $q$ is smooth, $[T_q] =[T_P] -[q^*T_X]$. Hence,
 \[ N_f=[i^*T_P] -[T_Y] -i^*([T_P] -[q^*T_X])=[f^*T_X] -[T_Y].\]
 
When $X$ and $Y$ are abelian varieties, the tangent bundles $T_X$ and $T_Y$ are trivial. Thus, $f^*T_X$ is also trivial. It follows from the properties of $\widetilde{Td}_\tau^{-1}$ that 
\begin{equation}\label{eq:pullback}
\begin{array}{rl}
f^*_{(\tau)}=Td_\tau^{-1}(N_f)\cdot f^*&=Td_\tau^{-1}(f^*T_X)\cdot Td_\tau^{-1}(-T_Y)\cdot f^*\\
&=1_{A^*(Y)}\cdot (Td_\tau^{-1}(T_Y))^{-1}\cdot f^*\\
&=1_{A^*(Y)}\cdot f^*\,=\, f^*. 
\end{array}
\end{equation}
Note that, if $X$ is an abelian variety and $\delta_X:X\to X\times X$ is the diagonal map, then $Td_\tau^{-1}(N_{\delta_X})=1_{A^*(X)}$ since $X\times X$ is also an abelian variety and $\delta_X$ is an $\lci$ morphism. Thus, we obtain the following:
\begin{lemma}\label{lem:product}
For an abelian variety $X$, there is a ring isomorphism $A^*_{(\tau)}(X)\iso A^*(X)$.
\end{lemma}
\begin{proof}
Since $A^*(X)$ and $A^*_{(\tau)}(X)$ coincide as groups, we only need to verify that the identity map is compatible with the multiplicative structure.  Let $\cdot$ denote the product in $A^*(X)$ and $\cdot_{\tau}$ denote the product in $A^*_{(\tau)}(X)$. Then, for $\alpha, \beta\in  A^*(X)$,
\[\alpha\cdot_{\tau}\beta=Td_\tau^{-1}(N_{\delta_X})\cdot \alpha\cdot\beta=\alpha\cdot\beta. \]
Thus the map $A^*_{(\tau)}\to A^*$ identifying the two groups is a ring isomorphism as well.
\end{proof}

Let $\Z[\bt] :=\Z[t_1,\ldots, t_n, \ldots ]$ be the graded ring of polynomials on variables $t_i, i>0$, of degree $-i$. We may form an oriented cohomology theory $A[\bt]^*$ from $A^*$ by extension of scalars by defining $A[\bt]^*(X):=A^*(X)\otimes_\Z \Z[\bt]$. Now, consider the oriented cohomology theory $CH_\Q[\mathbf{t}]_{(\mathbf{t})}$ defined in this way. We have the following theorem from \cite{LM}:
\begin{theorem}\label{theorem:twisted}
Let $k$ be a field of characteristic $0$. Then, the canonical morphism of oriented cohomology theories induces an isomorphism over $\Q$:
\[ \nu_{CH[\bt]^{(\bt)}}\otimes\Q : \Omega^*_\Q\to CH_\Q[\mathbf{t}]^*_{(\mathbf{t})}. \]
where $\nu_{CH[\bt]_{(\bt)}}$ is the canonical morphism given by universality of $\Omega^*$.
\end{theorem}
In the rest of the article, we will write $\nu$ to denote $\nu_{CH[\bt]^{(\bt)}}\otimes\Q$.
Combining Theorem~\ref{theorem:twisted} and Lemma~\ref{lem:product}, we get
\begin{theorem}\label{theorem:iso}
Let $X$ be an abelian variety over a field $k$ of characteristic $0$. Then, we have an  isomorphism of rings
\[ \psi_X:\Omega^*_\Q(X)\to CH_\Q[\mathbf{t}]^*(X).\]
\end{theorem}

\section{Recollection of $A$-motives}
In \cite[\S~5-6]{NZ}, for an oriented cohomology theory $A^*$ on $\Sm_k$, Nenashev and Zainoulline constructed the $A$-motive of a smooth projective variety $X$ over $k$ , following the ideas of \cite{Manin}. We briefly recall its construction.

\subsection{$A$-correspondences}
Let $X$ and $Y$ be smooth projective varieties over an algebraically closed field $k$ of characteristic $0$. We recall some facts about the category of \emph{$A$-correspondences} from \cite{NZ}. Given an oriented cohomology theory $A^*$, we define the category of $A$-correspondences, denoted $Cor_A$, as
\begin{itemize}
\item $Ob(Cor_A):=Ob(\SmProj_k)$;
\item $\Hom_{Cor_A}(X,Y):=A^*(X\times Y)$;
\item the composition of morphisms $\alpha\in A^*(X\times Y)$ and $\beta\in A^*(Y\times Z)$ is the correspondence 
\[\beta\circ\alpha := (p_{XZ})_*(p_{XY}^*(\alpha)\cdot p_{YZ}^*(\beta)) \;\in A^*(X\times Z). \]where $p_{XZ}$, $p_{XY}$ and $p_{YZ}$ are the respective projections from $X\times Y\times Z$.
\end{itemize}
There is a functor $c_A:\SmProj_k^{op}\to Cor_A$ given by $c_A(X)=X$ and $c_A(f)=(\Gamma_f)_*(1_{A(X)})\in A^*(Y\times X)$ for a morphism $f:X\to Y$, where $\Gamma_f:X\longby{(f,\id)}Y\times X$ is the graph morphism. For $\alpha\in A^*(X\times Y)$, we have the transpose $\alpha^t:=\iota^*(\alpha)\in A^*(Y\times X)$, where $\iota:Y\times X\to X\times Y$ is given by swapping the variables. 

For a correspondence $\alpha\in \Hom_{Cor_A}(Y,X)$, we define its realization $\sR_A(\alpha):A^*(Y)\to A^*(X)$ as follows: identify $A^*(Y)$ with $\Hom_{Cor_A}(pt,Y)$ and note that $\alpha$ defines a map $\Hom_{Cor_A}(pt,Y)\to \Hom_{Cor_A}(pt,X)$ given by composition with $\alpha$. This defines the map $\sR_A(\alpha)$ as 
\[ \beta\mapsto p_{X*}(\alpha\cdot p_Y^*\beta), \]
where $p_X$ and $p_Y$ are the respective projections to $X$ and $Y$ from $Y\times X$. We will denote $\sR_A$ by $\sR$ when there is no confusion. Note that the projection formula for the oriented cohomology theory $A^*$ implies that $\sR(c_A(f))=f^*$, so that the functor $A^*:\SmProj_k^{op}\to \Ab^\Z$ factors through $Cor_A$.

If $\alpha\in A^*(X\times Y)$ and $\beta\in A^*(Y\times Z)$, it follows from the definition that
\begin{equation}\label{eq:compositionR}
\sR(\beta)\circ \sR(\alpha)=\sR(\beta\circ\alpha)=\sR\big((p_{XZ})_*(p_{XY}^*(\alpha)\cdot p_{YZ}^*(\beta))\big),
\end{equation}

Also, using the projection formula, we get that, for $f:X\to Y$, $\alpha\in A^*(Z\times  Y)$ and $\beta\in A^*(X\times Z)$,
\begin{equation}\label{eq:graphcomp}
c_A(f)\circ\alpha=(\id_Z\times f)^*(\alpha)\quad\text{ and }\quad\beta\circ c_A(f)=(f\times \id_Z)_*(\beta).
\end{equation}
Applying transpose, we also get that for $\gamma\in A^*(Y\times Z)$ and $\delta\in A^*(Z\times X)$,
\begin{equation}\label{eq:graphcompT}
\gamma\circ c_A(f)^t=(f\times \id_Z)^*(\gamma)\quad\text{ and }\quad c_A(f)^t\circ\delta=(\id_Z\times f)_*(\delta).
\end{equation}
The grading on $A^*$ induces a grading on $\Hom_{Cor_A}$ given as \[\Hom_{Cor_A}^n(X,Y):=\oplus_iA^{n+d_i}(X_i\times Y), \] where $X_i$ are the irreducible components of $X$ and $d_i=\dim X_i$, making $\Hom_{Cor_A}$ into a graded algebra under composition. $Cor_A$ forms an additive category by defining $X\oplus Y=X\coprod Y$.

\subsection{$A$-motives}
\label{subsec:motdef}
\begin{definition}
Consider the category $Cor_A^0$ with thee same objects as $Cor_A$ and $\Hom_{Cor_A^0}(X,Y):= \Hom_{Cor_A}^0(X,Y)$. The pseudo-abelian completion of $Cor_A^0$ is called the \emph{category of effective $A$-motives}, denoted by $\mc{M}_A^\eff$. That is, the objects in $\mc{M}_A^\eff$ are pairs $(X,p)$ where $X\in Ob(Cor_A)$ and $p\in \Hom_{Cor_A^0}(X,X)$ is a projector (that is, $p\circ p=p$), and
\[ \Hom_{\sM_A^\eff}((X,p),(Y,q)) = \dfrac{\{\alpha\in\Hom_{Cor_A^0}(X,Y)\vert\alpha\circ p=q\circ\alpha\}}{\{\alpha\in\Hom_{Cor_A^0}(X,Y)\vert\alpha\circ p=q\circ\alpha=0\}}. \]
The \emph{category of $A$-motives}, denoted by $\mc{M}_A$, has as objects triplets $(X,p,m)$ where $(X,p)$ is an object in $\mc{M}_A^\eff$ and $m\in\Z$. The morphisms are defined as:
\[ \Hom_{\sM_A}((X,p,m),(Y,q,n)) = \dfrac{\{\alpha\in\Hom_{Cor_A}^{n-m}(X,Y)\vert\alpha\circ p=q\circ\alpha\}}{\{\alpha\in\Hom_{Cor_A}^{n-m}(X,Y)\vert\alpha\circ p=q\circ\alpha=0\}}. \]
\end{definition}
Note that this means $\id_{(X,p,0)}=\id_X=p\in\Hom_{\sM_A}((X,p,0),(X,p,0))$. The motive $(X,\id_X,0)$ is called the motive of $X$ and denoted by $h_A(X)$. The additive structure of $Cor_A$ induces a direct sum in the category $\mc{M}_A$.

Following \cite[\S~6]{Manin}, we call an irreducible variety $X$ in $\SmProj_k$ to be $A$-\emph{special} if it has a $k$-point, and for any morphism $e:\Spec k\to X$, $e_*(1_k)\in A^*(X)$ is independent of $e$. For an $A$-special variety $X$ of dimension $d$, we define the projectors $p^X_0$ and $p^X_{2d}$ in $\Hom_{Cor_A^0}(X,X)$ by the formulas
\begin{equation}\label{eq:p02d}
p^X_0 = e_*(1_k)\times 1_X\quad\text{ and }\quad p^X_{2d}=1_X\times e_*(1_k).
\end{equation}  
We  denote $h^i_A(X)=(X,p^X_i,0)$ for $i=0,2d$.

\section{Fourier-Mukai operator}
\label{sec:definition}

Let $A$ be an abelian variety of dimension $g$ and let $\hat{A}$ be its dual abelian variety. Denote by $\mc{P}$ the normalized Poincar\'e bundle on $A\times\hat{A}$. Here, ``normalized" means that $i^*\mc{P}$ and $\hat{i}^*\mc{P}$ are trivial, where $i:0\times\hat{A}\to A\times\hat{A}$ and $\hat{i}:A\times\hat{0}\to A\times\hat{A}$ are inclusions.

We wish to define an operator $\mc{F}:\Omega^*_\Q(A)\to\Omega^*_\Q(\hat{A})$ which has the usual properties of the Fourier-Mukai transformation on Chow rings or K-theory (see \cite{Beauville1}). 

Note that there is a Fourier-Mukai transformation on $CH_\Q[\mathbf{t}]$, that is induced from the one on $CH_\Q$. This is defined as the map 
\[ \sF_A^{CH}=\sR_{CH_\Q[\mathbf{t}]}(ch(\sP)):CH_\Q[\mathbf{t}](A)\longrightarrow CH_\Q[\mathbf{t}](\hat{A})\]
where by abuse of notation, $ch(\sP)$ denotes the Chern character of $\sP$ considered as an element in $CH_\Q[\mathbf{t}]^*(A\times\hat{A})$ by extension of scalars.

We now imitate this to define a Fourier-Mukai operator on $\Omega^*_\Q$. By \cite[Lemma~4.1.29]{LM}, there is a unique power series $l_{\bL}(u)=\sum_{i\geq 0}b_iu^{i+1}\in\bL\otimes\Q[[u]]$ with $b_0=1$ called the \emph{logarithm} such that $l_{\bL}(F_\Omega(u,v))=l_{\bL}(u)+l_{\bL}(v)$. Define $G:=\exp\circ l_{\bL}\in\bL\otimes\Q[[u]]$, where $\exp$ denote the exponential power series, $\exp(u)=\disp\sum_{i\geq 0}\dfrac{u^i}{i!}$. Note that, $G$ is a power series such that $G(F_\Omega(u,v))=G(u)G(v)$.
\begin{definition}
We define the Fourier-Mukai operator $\sF_A^\Omega :\Omega^*_\Q(A)\to\Omega^*_\Q(\hat{A})$ to be  $\sF_A^\Omega :=\sR_\Omega(G(c_1^\Omega(\sP)))$.
The dual operator $\hat{\sF}_A^\Omega$ is defined to be $\sR_
\Omega(G(c_1^\Omega(\sP))^t)$. When there is no confusion, we will denote $\sF_A^\Omega$ and $\hat{\sF}_A^\Omega$ by $\sF^\Omega$ and $\hat{\sF}^\Omega$ respectively.
\end{definition}

$\sF^\Omega$ is related to $\sF^{CH}$ in the following way:
\begin{proposition}
We have
\[ \psi_{\hat{A}}\circ\sF^\Omega=\sF^{CH}\circ\psi_A. \]
\end{proposition}
\begin{proof}
Note that in $CH_\Q[\mathbf{t}]^*_{(\mathbf{t})}$, $c^{(\bt)}_1(L)=\lambda_{(\bt)}(c^{CH}_1(L))$ where $\lambda_{(\bt)}\in \Z[\bt][[u]]$ is given as $\lambda_{(\bt)}(u)=u+\sum_{i\geq 2}t_{i-1}u^i$. This means that there is a power series $l_{(\bt)}$ such that $l_{(\bt)}(\lambda_{(\bt)}(u))=u$. Thus, $l_{(\bt)}(c^{(\bt)}_1(L))=c^{CH}_1(L)$. This implies that $l_{(\bt)}$ is the logarithm of the formal group law of $CH_\Q[\mathbf{t}]^*_{(\mathbf{t})}$. Also, note that, since $\nu$ is a morphism of oriented cohomology theories, we get $\nu(F_\Omega)=F_{(\bt)}$.
Thus,
\[\nu(l_{\bL}(F_\Omega(x,y)))=\nu(l_{\bL}(x))+\nu(l_{\bL}(y))\Rightarrow \nu(l_{\bL})(F_{(\bt)}(\nu(x),\nu(y))=\nu(l_{\bL})(\nu(x))+\nu(l_{\bL})(\nu(y)). \]
Since $\nu$ is an isomorphism and $l_{(\bt)}$ is unique, $\nu(l_{\bL})=l_{(\bt)}$. Hence, $\nu(G)=\exp\circ l_{(\bt)}$.

Denote by $\psi'_A$ the ring isomorphism $ CH_\Q[\mathbf{t}]^*_{(\mathbf{t})}(A)\to CH_\Q[\mathbf{t}]^*(A)$ that is the identity on elements, so that $\psi_A=\psi'_A\circ\nu$. $\psi'_A$   commutes with pushforwards and with the pullbacks of maps between abelian varieties.

Let $p_A$ and $p_{\hat{A}}$ be the respective projections of $A\times \hat{A}$ to $A$ and $\hat{A} $. 
Then, for $x\in\Omega_\Q(A)$, 
\begin{multline*} 
\psi_{\hat{A}}\circ\sF^\Omega(x) =\psi'_{\hat{A}}\circ\nu(p_{\hat{A}*}(p_A^*x\cdot G(c_1^\Omega(\sP)))=\psi'_{\hat{A}}(p_{\hat{A}*}(p_A^*\nu(x)\cdot \exp\circ l_{(\bt)}(c_1^{(\bt)}(\sP)))\\
=p_{\hat{A}*}(\psi'_{A\times\hat{A}}(p_A^*\nu(x))\cdot \psi'_{A\times\hat{A}}(\exp(c_1^{CH}(\sP))))=p_{\hat{A}*}(p_A^*\psi_A(x)\cdot ch(\sP))=\sF^{CH}\circ\psi_A(x), 
\end{multline*} which finishes the proof.
\end{proof}

Also, note that if $\sF$ denotes the Fourier-mukai transform on $CH_\Q$ defined in \cite{Beauville1}, then for $x=\disp\sum_{I=(n_1,\ldots,n_r,\ldots)} x_I\bt^I\in CH_\Q[\bt](A)$, we have 
\begin{equation}\label{eq:FCH}
\sF^{CH} (x) = \sum_{I=(n_1,\ldots,n_r,\ldots)}\sF(x_I) \bt^I.
\end{equation}
  Here, $\bt^I=t_1^{n_1}\cdots t_r^{n_r}\cdots$, all but finitely many $n_r$s being zero.

\section{Properties of the Fourier-Mukai operator}

Let $A$ be an abelian variety of  dimension $g$.  We denote by $\star$  the Pontrjagin product on $\Omega^*_\Q$, that is for $x,y\in \oq(A)$, we define \[ x\star y:= \mu_*(p_1^*x\cdot p_2^*y),\]
where $\mu$ is the multiplication on the abelian variety $A$, $p_1$ and $p_2$ are the projections of $A\times  A$ onto the first and second factors, and $\cdot$ is the usual product on $\oq$.

The Fourier-Mukai operator on $\Omega_\Q$ of abelian varieties  has the following properties:
\begin{proposition}
\label{prop:properties}
\begin{enumerate}
\item Let $\sigma_A$ denote the endomorphism of the abelian variety $A$ given by multiplication by $-1$. Then,
 \[\hat{\sF}^\Omega\circ\sF^\Omega = (-1)^g(\sigma_A)^* \quad \text{ and }\quad
\sF^\Omega\circ\hat{\sF}^\Omega = (-1)^g(\sigma_{\hat{A}})^*.\] 
\item For $x, y\in \Omega_\Q(A)$, we have
\[ \sF^\Omega(x\star y)=\sF^\Omega(x)\sF^\Omega(y) \quad \text{ and } \quad \sF^\Omega(xy)=(-1)^g\sF^\Omega(x)\star\sF^\Omega(y). \]
\item Let $f:A\to B$ be an isogeny of abelian varieties, and $\hat{f}:\hat{B}\to\hat{A}$ be the dual isogeny. Then we have
\[ \sF_B^\Omega\circ f_*=\hat{f}^*\circ\sF_A^\Omega \quad \text{ and }\quad \sF_A^\Omega\circ f^*=\hat{f}_*\circ \sF_B^\Omega. \]
\item Let $x\in \Omega_\Q^p(A)$. Write $\sF^\Omega(x)=\disp\sum_{q\leq g}y_q$, where $y_q\in \Omega_\Q^q(\hat{A})$. Then, we have, for $n\in\Z$,
\[ \bn^*y_q=n^{g-p+q}y_q, \]
where $\bn$ denotes the multiplication by $n$ on $\hat{A}$.
\end{enumerate}
\end{proposition}
\begin{proof}
For (1), we first check that $\hat{\sF}^{CH}\circ\sF^{CH} = (-1)^g(\sigma_A)^*_{CH}$. This follows readily from \eqref{eq:FCH} and \cite[Proposition~3]{Beauville1}. Thus,
\begin{align*}
\psi_A\circ\hat{\sF}^\Omega\circ\sF^\Omega &=\hat{\sF}^{CH}\circ\sF^{CH}\circ\psi_A \\
&= (-1)^g(\sigma_A)^*_{CH}\circ\psi_A=(-1)^g\psi_A\circ(\sigma_A)^*_\Omega.
\end{align*}
Since $\psi_A$ is an isomorphism, we get the desired result. The other part can be shown similarly.

To show (2), we check that $\psi_A(x\star y)=\psi_A(x)\star\psi_A(y)$. Indeed, $\psi_A$ is a ring homomorphism and commutes with pushforwards and with the pullbacks of maps between abelian varieties. As in the previous part, the result follows by applying $\psi_{\hat{A}}$ to both sides of the desired equality and noticing that the result holds for $\sF_{CH}$ by \cite[Proposition~3]{Beauville1}. (3) also follows similarly.

For (4), note that, if $x\in \Omega_\Q^p(A)$, then clearly, $\psi_{\hat{A}}\circ\sF^\Omega(x)=\sum_{q\leq g}\psi_{\hat{A}}(y_q)$, where $\psi_{\hat{A}}(y_q)\in CH_\Q[\bt]^q(\hat{A})$. Since $\psi_{\hat{A}}\circ\sF^\Omega(x)=\sF^{CH}(\psi_A(x))$, we get $\bn^*_{CH}\psi_{\hat{A}}(y_q)=n^{g-p+q}\psi_{\hat{A}}(y_q)$ by \cite[F3]{Beauville2} which implies
$\bn^*y_q=n^{g-p+q}y_q$ as required.
 
\end{proof}

\section{Beauville decomposition for algebraic cobordism}

We follow the ideas in \cite{Beauville2} to give a decomposition of $\Omega^*_\Q(A)$ into eigenspaces of $\bn^*$ using the Fourier-Mukai operator defined in \S~\ref{sec:definition}. 

For $s\in\Z$ and $A$ an abelian variety of dimension $g$ over $k$, let $\Omega^{\phantom{\Q}p}_{\Q s}(A)$ denote the sub-group
\[ \Omega^{\phantom{\Q}p}_{\Q s}(A):=\big\{ x\in \Omega^{p}_\Q(A) \vert\, \bn^*x=n^{2p-s}x \big\}. \]

Following the sketch of \cite[Proposition~1]{Beauville2} gives us the following.
\begin{proposition}\label{prop:decomp}
Let $x\in \Omega^{p}_\Q(A)$, and $m$ be any integer other than $0$, $1$ or $-1$. The following conditions  are equivalent:
\begin{enumerate}
\item $\sF^\Omega(x)\in \Omega^{g-p+s}_\Q(\hat{A})$;
\item $x\in \Omega^{\phantom{\Q}p}_{\Q s}(A)$;
\item $\bm^*x=m^{2p-s}x$;
\item $\bm_*x=m^{2g-2p+s}x$;
\item $\sF^\Omega(x)\in \Omega^{\phantom{\Q}g-p+s}_{\Q s}(\hat{A})$.
\end{enumerate}
\end{proposition}

\begin{proof}
\begin{description}
\item[(1)$\Rightarrow$(2)] Let $y=(-1)^g(\sigma_{\hat{A}})^*\sF^\Omega(x)$ and let $\hat{\sF}^\Omega(y)=\disp\sum_{q\leq g}x_q$ with $x_q\in \Omega^{q}_\Q(A)$. Then, Proposition~\ref{prop:properties}, part (4) gives us
\[ \bn^*x_q=n^{g-(g-p+s)+q}x_q=n^{p+q-s}x_q. \]
But now, \begin{align*}
\hat{\sF}^\Omega(y)&=(-1)^g\hat{\sF}^\Omega\circ(\sigma_{\hat{A}})^*\circ\sF^\Omega(x)\\
&=(-1)^g(\sigma_{A})^*\circ\hat{\sF}^\Omega\circ\sF^\Omega(x)\quad[\text{By Proposition~\ref{prop:properties}, part (3)}]\\
&=x,\quad[\text{By Proposition~\ref{prop:properties}, part (1)}]
\end{align*}
Then, $x=x_p$ and $\bn^*x=n^{2p-s}x$,
thus showing that $x\in \Omega^{\phantom{\Q}p}_{\Q s}(A)$.
\item[(2)$\Rightarrow$(3)] This is by definition.
\item[(3)$\Rightarrow$(4)] Since $\bm$ is a surjective endomorphism of $A$ with finite kernel of size $m^{2g}$, we have $\bm_*^{CH}\bm^*_{CH}=m^{2g}\cdot \id_{CH}$. Then, $\psi_A(\bm_*\bm^*x)=\bm_*^{CH}\bm^*_{CH}\psi_A(x)=m^{2g}\psi_A(x)$ implies $\psi_A(\bm_*x)=\psi_A(m^{2g-2p+s}x)$ which gives the result by the injectivity of $\psi_A$ since $m\neq 0,\pm 1$.
\item[(4)$\Rightarrow$(5)] By Proposition~\ref{prop:properties}, part (3), we get that 
\[ \bm^*\sF^\Omega(x) = \sF^\Omega(\bm_*x) = m^{2g-2p+s}\sF^\Omega(x) =m^{g-p+(g-p+s)}\sF^\Omega(x).\]
Since $m\neq 0,\pm 1$, this implies by Proposition~\ref{prop:properties}, part (4), that $\sF^\Omega(x)\in \Omega^{g-p+s}_\Q(\hat{A})$, but then, by definition, $\sF^\Omega(x)\in \Omega^{\phantom{\Q}g-p+s}_{\Q s}(\hat{A})$.
\item[(5)$\Rightarrow$(1)] This is obvious.
\end{description}
\end{proof}

\begin{theorem}\label{theorem:Beauville} Let $A$ be an abelian variety of dimension $g$ over $k$. Then, we have
\[ \Omega_\Q^{p}(A) = \bigoplus_{s=2p-2g}^{\mathrm{Min}(2p,p)} \Omega^{\phantom{\Q}p}_{\Q s}(A). \]
\end{theorem}

\begin{proof}
Let $x\in \Omega_\Q^{p}(A)$ and let $y=\sF^\Omega(x)$. We can write $y=\sum_{q\leq g}y_q$, with $y_q\in \Omega_\Q^{q}(\hat{A})$. By Proposition~\ref{prop:properties}, part (4), $\bn^*y_q = n^{g-p+q}y_q$. That is, $y_q\in \Omega^{\phantom{\Q}q}_{\Q p+q-g}(\hat{A})$. Then, Proposition~\ref{prop:decomp} gives us that $\hat{\sF}^\Omega(y_q)\in \Omega^{\phantom{\Q}p}_{\Q p+q-g}(A)$. But, \[(-1)^g(\sigma_{A})^*x=\hat{\sF}^\Omega(y)=
\sum_{q\leq g}\hat{\sF}^\Omega(y_q).\] 
Putting $x_s=(-1)^g(\sigma_{A})^*\hat{\sF}^\Omega(y_{s+g-p})$, we obtain $x=\sum_{s\leq p}x_s$. We now improve the limits of the sum. 

Since $\hat{\sF}(y_q)$ has degree $p$, and by definition, $\psi_A\circ\hat{\sF}(y_q)=\disp\sum_{i\geq 0} p_{1*}\big(p_2^*(\psi_{\hat{A}}(y_q))\dfrac{(c_1^{CH}(\sP))^i}{i!}\big)$, we get 
\[ \psi_A\circ\hat{\sF}(y_q)=p_{1*}\big(p_2^*(\psi_{\hat{A}}(y_q))\dfrac{(c_1^{CH}(\sP))^{p+g-q}}{(p+g-q)!}\big). \]
Also, since $y_q$ has degree $q$, we get
\[ \psi_{\hat{A}}(y_q)=p_{2*}\big(p_1^*(\psi_A(x))\dfrac{(c_1^{CH}(\sP))^{q+g-p}}{(q+g-p)!}\big). \]
Since $(c_1^{CH}(\sP))^{p+g-q}=0$ if $q>p+g$ and $(c_1^{CH}(\sP))^{q+g-p}=0$ if $q<p-g$, we get that $x_s=0$ if $s>2p$ or $s<2p-2g$.
\end{proof}

As an easy consequence of the definition of $\Omega^{\phantom{\Q}p}_{\Q s}(A)$ and Proposition~\ref{prop:decomp}, we get the following:
\begin{proposition}
\begin{enumerate}
\item $\sF^\Omega(\Omega^{\phantom{\Q}p}_{\Q s}(A))=\Omega^{\phantom{\Q}g-p+s}_{\Q s}(\hat{A})$.
\item If $f:A\to B$ is a homomorphism of abelian varieties of relative dimension $m$, then $f^*\Omega^{\phantom{\Q}p}_{\Q s}(B)\subset\Omega^{\phantom{\Q}p}_{\Q s}(A)$ and $f_*\Omega^{\phantom{\Q}p}_{\Q s}(A)\subset\Omega^{\phantom{\Q}p+m}_{\Q s}(B)$.
\item If $x\in\Omega^{\phantom{\Q}p}_{\Q s}(A),\,y\in\Omega^{\phantom{\Q}q}_{\Q t}(A)$, then $xy\in\Omega^{\phantom{\Q}p+q}_{\Q s+t}(A)$ and 
$x\star y\in\Omega^{\phantom{\Q}p+q-g}_{\Q s+t}(A)$, where $\star$ denotes the Pontryagin product on $A$.
\end{enumerate}
\end{proposition}
\begin{proof}
(1) is immediate from Proposition~\ref{prop:decomp}. 

(2) follows from the fact that $f\circ\bm=\bm\circ f$ and the equivalence of (3) and (4) in Proposition~\ref{prop:decomp}.

If $x\in\Omega^{\phantom{\Q}p}_{\Q s}(A),\,y\in\Omega^{\phantom{\Q}q}_{\Q t}(A)$, then $xy\in\Omega^{\phantom{\Q}p+q}_{\Q s+t}(A)$ by definition. Also, note that $\sF^\Omega(x\star y)=\sF^\Omega(x)\sF^\Omega(y)\in\Omega^{\phantom{\Q}2g-p-q+s+t}_{\Q s+t}(\hat{A})$. Applying $\hat{\sF}^\Omega$, we get $(\sigma_A)^*(x\star y) \in \Omega^{\phantom{\Q}p+q-g}_{\Q s+t}(A)$, which gives the result by part (2).
\end{proof}

\section{Consequences for numerically trivial cobordism cycles}
\label{sec:num}
Let $I\subset CH^g(A)$ denote the set of 0-cycles of degree 0 on $A$. In \cite[\S~4]{Bloch}, Bloch showed that
\begin{equation}\label{eq:Bloch}
I^{\star (r+1)}\star CH^r(A)=(0)
\end{equation}
in the cases $r=0,1,g-2,g-1,g$ where $g=\dim(A)$. In \cite{Beauville1}, Beauville conjectured that 
\[
\tag{$F_p$}\label{Fp}\text{ For all }x\in CH^p_\Q(A),\text{ we have }\sF(x)\in CH^{\geqslant g-p}_\Q(\hat{A}).
\]
He verified \eqref{Fp} for $p=0,1,g-2,g-1,g$ (\cite[Proposition~8.(i)]{Beauville1}) and also showed that  
\begin{proposition}[{\cite[Proposition~9]{Beauville1}}]\label{prop:Beauville}
\eqref{Fp} implies that $I^{\star (p+1)}\star CH^p(A)=(0)$. In particular, the groups $I^{\star (g+1)}$, $I^{\star g}\star CH^{g-1}(A)$ and $I^{\star (g-1)}\star CH^{g-2}(A)$ are zero.
\end{proposition}

We prove an analogue of this Proposition replacing $I$ with numerically trivial cobordism cycles. A notion of numerical equivalence on $\Omega^*(X)$ was defined in \cite{BP}. We briefly recall the construction.
\begin{definition}\label{def:num}
Let $X$ be a smooth projective scheme over a field $k$ of characteristic 0. Consider the composition of maps 
\[
\Omega^m(X)\otimes\Omega^n(X)\to\Omega^{m+n}(X)\by{\pi_*}\Omega^{m+n-\dim_X}(k),
\]  
where  $\pi$ is the structure morphism $X\to\Spec(k)$.
This gives a map of $\bL$-modules $\Omega^*(X)\longrightarrow \Hom_{\bL}(\Omega^*(X),\Omega^*(k))$. 
We say that a cobordism cycle in $\Omega^*(X)$ is {\em numerically equivalent to} 0 if it is in the kernel of this map:
\[ \sN^*(X):=\ker\Big(\Omega^*(X)\rightarrow \Hom_{\bL}(\Omega^*(X),\Omega^*(k))\Big) \] and \[\Omega^*_{\n}(X):=\Omega^*(X)/\sN^*(X). \]
\end{definition}

Let $A$ be an abelian variety over $k$ of dimension $g$. 
\begin{lemma}\label{lemma:num}
$\sF^\Omega$ carries $\sN^*(A)$ to $\sN^*(\hat{A})$.
\end{lemma}
\begin{proof}
Let $\alpha\in \sN^*(A)$. Then, by definition, $\sF^\Omega(\alpha)=p_{\hat{A}*}(p_A^*\alpha\cdot G(c^\Omega_1(\sP)))$. Let $\pi_{\hat{A}}$ and $\pi_A$ be the structure morphisms of $\hat{A}$ and $A$ respectively and let $p_{\hat{A}}$ and $p_A$ be the respective projections of $A\times\hat{A}$ to $\hat{A}$ and $A$. We obtain, by the projection formula, for any $\gamma\in\Omega^*_\Q(\hat{A})$,
\begin{align*}
\pi_{\hat{A}*}(\sF^\Omega(\alpha)\cdot\gamma)&=\pi_{\hat{A}*}p_{\hat{A}*}\left(p_A^*\alpha\cdot G(c^\Omega_1(\sP))\cdot p_{\hat{A}}^*\gamma\right)\\
&=\pi_{A*}p_{A*}\left(p_A^*\alpha\cdot G(c^\Omega_1(\sP))\cdot p_{\hat{A}}^*\gamma\right)\\
&=\pi_{A*}\left(\alpha\cdot\hat{\sF}^\Omega(\gamma)\right).
\end{align*}
Thus, numerically triviality of $\alpha$ implies that $\sF^\Omega(\alpha)$ is numerically trivial.
\end{proof}
\begin{proposition}\label{prop:FpOmega}
Fix $0\leq p\leq g$. If $x\in\Omega^p_{\Q}(A)$ is such that $\sF^\Omega(x)\in \Omega^{\geqslant g-p}_{\Q}(A)$. Then, $\sN^{*\star (p+1)}_\Q\star x=0$. 
\end{proposition}
\begin{proof}
Pick $\alpha_1,\ldots,\alpha_{p+1}\in \sN^*_\Q(A)$. Note that by Proposition~\ref{prop:properties}, part (2), 
\begin{equation}\label{eq:pontryagin}
\sF^\Omega(\alpha_1\star\alpha_2\star\cdots\star\alpha_{p+1})=\sF^\Omega(\alpha_1)\sF^\Omega(\alpha_2)\cdots\sF^\Omega(\alpha_{p+1}).
\end{equation}

Suppose that, for some $i$, $\sF^\Omega(\alpha_i)\in\Omega^{\leqslant 0}_{\Q}(\hat{A})$, then by the Generalized degree formula (\cite[Theorem~4.4.7]{LM}), we get
\[ \sF^\Omega(\alpha_i)=\deg(\sF^\Omega(\alpha_i))[\Id_{\hat{A}}]+\sum_{\codim_{\hat{A}}Z>0}\omega_Z[\tilde{Z}\to {\hat{A}}], \]
where the sum is over closed integral subschemes $Z\subset \hat{A}$, $\tilde{Z}$ is smooth with a birational morphism $\tilde{Z}\to Z$ and $\omega_Z\in \bL^{<0}$. Lemma~\ref{lemma:num} shows that $\sF^\Omega(\alpha_i)\in\sN^*_\Q(\hat{A})$. Then, by \cite[Proposition~3.4]{BP}, $\deg(\sF^\Omega(\alpha_i))=0$. Hence, $\sF^\Omega(\alpha_i)\in \bL^{<0}\cdot\Omega^{\geqslant 1}(\hat{A})$.

By \eqref{eq:pontryagin}, it follows that $\sF^\Omega(\alpha_1\star\cdots\star\alpha_{p+1})$ is in $\Omega^{\geqslant p+1}(\hat{A})$ or in $\bL^{<0}\cdot\Omega^{\geqslant p+1}(\hat{A})$. Thus, by Proposition~\ref{prop:properties}, part (2),
\[ \sF^\Omega(\alpha_1\star\cdots\star\alpha_{p+1}\star x)=\sF^\Omega(\alpha_1\star\cdots\star\alpha_{p+1})\sF^\Omega(x)=0. \]
Applying $\hat{\sF}^\Omega$ and using Proposition~\ref{prop:properties}, part (1), we get
\[ (-1)^g(\sigma_A)^*(\alpha_1\star\cdots\star\alpha_{p+1}\star x)=0. \]
Hence, $\alpha_1\star\cdots\star\alpha_{p+1}\star x=0$, which completes the proof.
\end{proof}

The same arguments as in the proof of the above proposition with $p=g$ shows  
\begin{corollary}
$\sN^{*\star (g+1)}_\Q=(0)$.
\end{corollary}

One can check that $\sN^*_\Q(A)$ forms an ideal of $\Omega^*_\Q(A)$ under Pontryagin product. By \cite[Lemma~4.5.10]{LM} and \cite[Theorem~3.2]{BP}, $\sN^g_{\Q}(A)\iso I$, which is the subgroup of $0$-cycles of degree $0$. This implies 
$\sN^g_{\Q}(A)/\sN^g_{\Q}(A)^{\star 2}\iso I/I^{\star 2}\iso A$. It would be interesting to study the structure of the group 
$\sN_{\Q}(A)/\sN_{\Q}(A)^{\star 2}$.

\section{Motivic decomposition}
\label{sec:motive}

Our goal in this section is to get a canonical decomposition of cobordism motives  of abelian varieties as shown by Deninger and Murre in \cite{DM} for Chow motives. That is, a decomposition
\[ h_\Omega(A) = \bigoplus_i h^i_\Omega(A) \quad \text{where } h^i_\Omega(A)=(A,\pi_i,0),\]
$\pi_i$ being orthogonal  projectors such that $\bn^*\pi_i=n^i\pi_i$. In \cite[\S~5]{Scholl}, Scholl gave an alternative proof of the decomposition for Chow motives and also described the projectors in the decomposition more explicitly. 

Let $A$ be an abelian variety of dimension $g$ over $k$. Let $\Delta$ denote the class of the diagonal morphism $[A\to A\times A]$ in $\Omega^g_\Q(A\times A)$. We are going to show
\begin{theorem}\label{theorem:decomp}
There is a canonical decomposition
\[\Delta=\sum_{i=0}^{2g}\pi_i\text{ in }\Omega^g_\Q(A\times A)\]
such that $(\id_A\times \bn)^*\pi_i=n^i\pi_i$ for all $n\in\Z$ and $\pi_i$'s are mutually orthogonal projectors, that is, $\pi_i^2=\pi_i$ and $\pi_i\circ\pi_j=0$ for $i\neq j$. Also, $c_\Omega(\bn)\circ\pi_i=n^i\pi_i=\pi_i\circ c_\Omega(\bn)$.
\end{theorem}

\begin{proof}
Note that such a decomposition is unique if it exists. Indeed, if $\{\rho_i\}_{i=0}^{2g}$ is another such decomposition, then
$ \pi_i=\sum_{j=0}^{2g}\pi_i\circ\rho_j$. Composing with $c_\Omega(\bn)$ from the left, we get $n^i\pi_i=\sum_{j=0}^{2g}n^j\pi_i\circ\rho_j$, which by substituting the expression for $\pi_i$ gives
\[ \sum_{j=0}^{2g}(n^j-n^i)\pi_i\circ\rho_j=0.\]
Since, this is true for all $n$, we must have $\pi_i\circ\rho_j=0$ for $i\neq j$, implying $\pi_i=\pi_i\circ\rho_i$. We can similarly show that $\rho_i=\pi_i\circ\rho_i$ implying $\pi_i=\rho_i$.

To see the existence, we first note that we have such a decomposition of the diagonal in $CH_\Q^g[\bt](A\times A)$ induced by extension of scalars from the canonical decomposition of Chow motives of $A$, as shown in \cite[\S~5]{Scholl}. That is, in $CH_\Q^g[\bt](A\times A)$, the diagonal $[\Delta]_{CH}$ may be expressed as
$ [\Delta]_{CH}=\sum_{i=0}^{2g}p^{\can}_i$ where for each $i$, $p_i^\can$ is a projector and $p_i^\can\circ p_j^\can=0\in \Hom_{Cor_{CH_\Q[\bt]^*}^0}(A,A)$ for $i\neq j$. Also, $c_{CH}(\bn)\circ p_i^\can=n^ip_i^\can=p_i^\can\circ c_{CH}(\bn)$. 

Now, take $\pi_i=\psi_{A\times\hat{A}}^{-1}(p_i^\can)$. Since, for $\alpha,\beta\in\oq(A\times A)$, we have $\psi_{A\times\hat{A}}(\alpha\circ\beta)=\psi_{A\times\hat{A}}(\alpha)\circ\psi_{A\times\hat{A}}(\beta)$, we readily get that $\pi_i$ is a projector and $\pi_i\circ\pi_j=0$ for $i\neq j$. Also, since, $\psi(c_\Omega(\bn))=c_{CH}(\bn)$, we get $c_\Omega(\bn)\circ\pi_i=n^i\pi_i=\pi_i\circ c_\Omega(\bn)$ for all $n\in\Z$.
\end{proof}

Let $\sM_\Omega$ be the category of cobordism motives over $k$, defined in \ref{subsec:motdef}. Let $\varphi$ denote the canonical morphism $\varphi=\nu_{CH}:\Omega^*\to CH^*$. $\varphi$ induces a functor  $\tilde{\varphi}:\sM_\Omega	\to\sM_{CH}$ of the corresponding categories of motives, acting as $\varphi$ on the morphisms and on objects as $(X,p,m)\mapsto (X,\varphi(p),m)$.
\begin{corollary}
There is a canonical decomposition of the cobordism motive of an abelian variety $A$ of dimension $g$ over $k$:
\[ h_\Omega(A)=\bigoplus_{i=0}^{2g}h_\Omega^i(A)\,, \]
where $h_\Omega(A)=(A,\id_A, 0)$ is the motive of $A$, $h_\Omega^i(A)=(A,\pi_i,0)$ and the $\pi_i$'s are such that $c_\Omega(\bn)\circ\pi_i=n^i\pi_i=\pi_i\circ c_\Omega(\bn)$. Furthermore, we have $\tilde{\varphi}(h_\Omega^i(A))=h^i_{CH}(A)^\can$, as defined in \cite[\S~5]{Scholl}.
\end{corollary}
\begin{proof}
This is immediate since, if $\id_A=\disp\sum_{i=0}^{n}\pi_i$ for mutually orthogonal projectors $\pi_i$, then
\[ (A,\id_A,0)=\bigoplus_{i=0}^n(A,\pi_i,0). \]
Also, it follows from the construction in the proof of Theorem~\ref{theorem:decomp} that $\varphi(\pi_i)=p_i^\can$ for all $i$, which implies $\tilde{\varphi}(h_\Omega^i(A))=h^i_{CH}(A)^\can$.
\end{proof}
\begin{remark}
Note that, by definition, the projectors $\pi_0$ and $\pi_{2g}$ are the same as those defined in \eqref{eq:p02d}, that is, $\pi_0=p^A_0$ and $\pi_{2g}=p^A_{2g}$. Also, $\pi_i^t=\pi_{2g-i}$ since this property holds for the $p_i^{\can}$s.
\end{remark}

\begin{ack}

This work was supported by the National Research Foundation of Korea (NRF) grant funded by the Korean government (MSIP) (No. 2013-042157). 

\end{ack}

\bibliography{biblo}{}

\providecommand{\bysame}{\leavevmode\hbox to3em{\hrulefill}\thinspace}
\providecommand{\MR}{\relax\ifhmode\unskip\space\fi MR }
\providecommand{\MRhref}[2]{%
  \href{http://www.ams.org/mathscinet-getitem?mr=#1}{#2}
}
\providecommand{\href}[2]{#2}
\begin{thebibliography}{10}

\bibitem{BP}
A.~Banerjee and J.~Park, \emph{{On numerical equivalence for algebraic
  cobordism}}, preprint available at \url{http://arxiv.org/abs/1312.1787},
  2013.

\bibitem{Beauville1}
A.~Beauville, \emph{{Quelques remarques sur la transformation de Fourier dans
  l’anneau de Chow d’une vari\'et\'e ab\'elienne}}, Algebraic geometry, LNM
  \textbf{1016} (1983), 238--260.

\bibitem{Beauville2}
\bysame, \emph{{Sur l'anneau de Chow d'une vari\'et\'e ab\'elienne}},
  Mathematische Annalen \textbf{273} (1986), 647--651.

\bibitem{Bloch}
Spencer Bloch, \emph{{Some Elementary Theorems about Algebraic Cycles on
  Abelian Varieties}}, Inventiones math. \textbf{37} (1976), 215--228.

\bibitem{DM}
C.~Deninger and J.~Murre, \emph{{Motivic decomposition of abelian schemes and
  the Fourier transform}}, J. Reine Angew. Math. \textbf{422} (1991), 201--219.

\bibitem{Fulton}
W.~Fulton, \emph{{Intersection theory}}, second edition ed., Ergebnisse der
  Math. Grenzgebiete, vol.~3, Springer-Verlag, Berlin, 1998.

\bibitem{Lazard}
M.~Lazard, \emph{{Sur les groupes de Lie formels \'a un param\`etre}}, Bull.
  Soc. Math. France \textbf{83} (1955), 251--274. \MR{0073925 (17,508e)}

\bibitem{LM}
M.~Levine and F.~Morel, \emph{{Algebraic Cobordism}}, {Springer Monographs
  Math.}, Berlin, 2007.

\bibitem{Manin}
Ju.~I. Manin, \emph{{Correspondences, motifs and monoidal transformations}},
  Math. USSR-Sb. \textbf{6} (1968), 439--470.

\bibitem{Mukai}
S.~Mukai, \emph{{Duality between $D(X)$ and $D(\hat{X})$ with its application
  to Picard sheaves}}, Nagoya Math. J. \textbf{81} (1981), 153--175.

\bibitem{NZ}
A.~Nenashev and K.~Zainoulline, \emph{{Oriented cohomology and motivic
  decompositions of relative cellular spaces}}, J. Pure Appl. Algebra
  \textbf{205} (2006), 323--340.

\bibitem{Quillen}
Daniel Quillen, \emph{Elementary proofs of some results of cobordism theory
  using {S}teenrod operations}, Advances in Math. \textbf{7} (1971), 29--56.
  \MR{0290382 (44 \#7566)}

\bibitem{Scholl}
A.~J. Scholl, \emph{{Classical motives}}, Motives, Proc. Sympos. Pure Math.
  \textbf{55} (1994), 163--187.

\end{thebibliography}
\bibliographystyle{amsplain}

\end{document}